\documentclass[12pt,a4paper]{amsart}
\usepackage{a4wide}
\usepackage[utf8]{inputenc}
\usepackage{amsmath,amssymb}

\usepackage{xcolor}
\usepackage{booktabs}
\usepackage{hyperref}
\hypersetup{
    colorlinks,
    citecolor=black,
    filecolor=black,
    linkcolor=blue,
    urlcolor=black
}
\usepackage{stmaryrd}
\usepackage{url}
\usepackage{longtable}
\usepackage[figuresright]{rotating}
\usepackage{amsthm}
\usepackage{enumerate}
\usepackage{geometry}
\newtheorem{definition}{Definition}
\newtheorem{theorem}[definition]{Theorem}
\newtheorem{proposition}[definition]{Proposition}

\newtheorem{lemma}[definition]{Lemma}

\newtheorem*{claim*}{Claim}

\newcommand{\0}{\emptyset}
\newcommand{\mc}{\mathcal}

\newcommand{\RR}{\mathbb{R}}

\newcommand{\QQ}{\mathbb{Q}}

\newcommand{\IFF}{\Leftrightarrow}
\newcommand{\IMP}{\Rightarrow}
\newcommand{\foralmostall}{\forall^\infty}
\newcommand{\existsinfty}{\exists^\infty}

\newcommand{\Ii}{\mc{I}}
\newcommand{\Jj}{\mc{J}}

\newcommand{\Mm}{\mc{M}}
\newcommand{\Nn}{\mc{N}}

\newcommand{\cm}{\mathfrak{c}}

\newcommand{\bez}{\backslash}
\newcommand{\se}{\subseteq}
\newcommand{\sen}{\subsetneq}
\newcommand{\es}{\supseteq}

\newcommand{\rest}{\restriction}

\newcommand{\baire}{\omega^\omega}

\newcommand{\concat}{^{\frown}}
\newcommand{\eps}{\varepsilon}
\newcommand{\id}{\tn{id}}

\newcommand{\rank}{\tn{rank}}

\newcommand{\tn}[1]{\textnormal{#1}}
\newcommand{\ti}[1]{\textit{#1}}

\newcommand{\add}{\tn{add}}
\newcommand{\cof}{\tn{cof}}

\newcommand{\supp}{\tn{supp}}
\newcommand{\Bor}{\tn{Bor}}

\newcommand{\Perf}{\textnormal{Perf}}

\newcommand{\splitt}{\textnormal{split}}
\newcommand{\Succ}{\textnormal{Succ}}
\newcommand{\succe}{\textnormal{succ}}
\newcommand{\stem}{\textnormal{stem}}
\newcommand{\tips}{\textnormal{tips}}

\def\cc{\mathcal{C}}

\def\cm{\mathcal{M}}
\def\cn{\mathcal{N}}

\def\bbr{\mathbb{R}}

\def\c{\mathfrak{c}}

\def\w{\omega}

\def\baire{\w^\w}

\def\with{{\smallfrown}}
\def\restricted{\upharpoonright}

\def\cof{\rm cof}

\def\then{\longrightarrow}

\def\force{\Vdash}

\title{Mycielski among trees}

\author{Marcin Michalski}
\email{marcin.k.michalski@pwr,edu.pl}
	
\author{Robert Rałowski}
\email{robert.ralowski@pwr.edu.pl}

\author{Szymon Żeberski}
\email{szymon.zeberski@pwr.edu.pl}

\thanks{The work has been partially financed by grant S50129/K1102 (0401/0052/18) from the Faculty of Fundamental Problems of Technology, Wrocław University of Science and Technology.}

\address{Marcin Michalski, Robert Rałowski, Szymon Żeberski, Department of Computer Science, Faculty of Fundamental Problems of Technology, Wrocław University of Technology,	50-370 Wrocław, Poland}

\date{}

\begin{document}

\begin{abstract}
	Two-dimensional version of the classical Mycielski theorem says that for every comeager or conull set $X\se [0,1]^2$ there exists a perfect set $P\se [0,1]$ such that $P\times P\se X\cup \Delta$. We consider generalizations of this theorem by replacing a perfect square with a rectangle $A\times B$, where $A$ and $B$ are bodies of other types of trees with $A\se B$. In particular, we show that for every comeager $G_\delta$ set $G\se \baire\times \baire$ there exist a Miller tree $M$ and a uniformly perfect tree $P\se M$ such that $[P]\times [M]\se G\cup\Delta$ and that $P$ cannot be a Miller tree. In the case of measure we show that for every subset $F$ of $2^{\omega}\times 2^\omega$ of full measure there exists a uniformly perfect tree $P\se 2^{<\omega}$ such that $[P]\times[P]\se F\cup\Delta$ and no side of such a rectangle can be a body of a Silver tree or a Miller tree. We also show some properties of forcing extensions of the real line from which we derive nonstandard proofs of Mycielski-like theorems via Shoenfield Absoluteness Theorem.
\end{abstract}

\maketitle

\section{Introduction and notation}

The motivation of this paper is the following two-dimensional version of classical Mycielski theorem (see \cite{Mycielski}).

\begin{theorem}
	For every comeager or conull set $X\se [0,1]^2$ there exists a perfect set $P\se [0,1]$ such that $P\times P\se X\cup \Delta$, where  $\Delta=\{(x,x): x\in [0,1]\}.$
\end{theorem}	
In the Cantor space $2^\omega$ and the Baire space $\omega^\omega$ each perfect set has a natural combinatorial description.
Let $A\in \{2,\omega\}$ and denote $A^{<\omega}=\bigcup_{n\in\omega}A^n$. Let us recall that $T\se A^{<\omega}$ is a tree on $A$  if for each $\sigma\in T$ and every $n\in\omega$ we have $\sigma\rest n\in T$.
A body of a tree $T\se A^{<\omega}$ is the set  ${[T]=\{x\in A^\omega: (\forall n)( x\rest n\in T)\}}$ of infinite branches of $T$.
A tree $T\se A^{<\omega}$ is called a perfect tree (or a Sacks tree), if
	\[
		(\forall \sigma\in T)( \exists \tau\in T)( \sigma\se \tau \land (\exists i, j\in A)(i\neq j\land \tau\concat i,\tau\concat j\in T)).
	\]
Then $P\subseteq A^\omega$ is a perfect set if and only if $P$ is a body of a perfect tree.
\\
A natural question arises whether we may replace perfect trees with another types of trees.
\\
Our general setup will be as follows. We will consider a subset $X$ of $2^\omega\times 2^\omega$ or $\baire\times \baire$, of full measure or comeager, and investigate whether there exist trees $T_1, T_2$ satisfying $T_2\se T_2$ such that $[T_1]\times [T_2]\se X\cup \Delta $, where $\Delta$ denotes a diagonal, i.e. ${\Delta=\{(x,x): x\in S\}}$ and $S$ is the space we work in. Natural examples of considered trees are Miller, Laver, uniformly perfect and Sliver trees.

We adopt the standard set-theoretical notation (see \cite{Jech}). Let $T\se A^{<\omega}$ be a tree on a set $A\in \{2,\omega\}$. We will use the following notions related to trees:
\begin{itemize}
	\item $\succe_T(\sigma)=\{a\in A: \sigma\concat a\in T\}$;
	\item $\splitt(T)=\{\sigma\in T: |\succe_T(\sigma)|\geq 2\}$;
	\item $\Succ_T(\sigma)=\{\tau\in\splitt(T): \sigma\se\tau  \tn{ and } \neg(\exists \tau'\in\splitt(T))(\sigma\se\tau'\subsetneq\tau)\}$;
	\item $\omega$-$\splitt(T)=\{\sigma\in T: |\succe_T(\sigma)|=\omega\}$.
\end{itemize}

A tree $T\se A^{<\omega}$ is called
\begin{itemize}
	\item a Miller or superperfect tree, if $(\forall \sigma\in T)( \exists \tau\in\omega$-$\splitt(T))( \sigma\se \tau)$;
	\item a Laver tree, if $(\exists \sigma)( \forall \tau\in T)(\tau\se\sigma \lor (\sigma\se\tau \land\tau\in\omega$-$\splitt(T)))$.
\end{itemize}
We will denote the shortest splitting node of a given tree $T$ by $\stem(T)$. Nodes $\tau, \sigma\in A^{<\omega}$ are orthogonal (denoted by $\sigma\perp\tau$), if neither $\tau\se\sigma$ nor $\sigma\se\tau$. Sometimes we will be indexing nodes with nodes. In such cases for the sake of brevity we will write e.g. $\tau_{010}$ instead of $\tau_{(0,1,0)}$.
\\
As mentioned above, we will also consider some specific types of perfect trees (see \cite{KorchWeiss}). We call a perfect tree $T\se A^{<\omega}$
\begin{itemize}
	\item uniformly perfect, if for every $n\in\omega$ either $A^n\cap T\se\splitt(T)$ or $A^n\cap \splitt(T)=\0$;
	\item a Silver tree, if $(\forall \sigma,\tau\in T)(|\sigma|=|\tau|\IMP(\forall a\in A)(\sigma\concat a \in T \IFF \tau\concat a\in T))$.
\end{itemize}
Before we proceed let us notice that to provide an example of a comeager subset of $X^2$ which does not contain a rectangle $A\times B$ of sets of certain type, it is enough to show that there exists comeager set $G\se X$ with $A\not\se G$ or $B\not\se G$. Indeed, in such a case $G\times X$ is comeager too (by Kuratowski-Ulam Theorem) and $A\times \{x\}\not\se G\times X$ for every $x\in X$. The same is true for the measure case thanks to Fubini Theorem. This observation gives weight to Propositions \ref{G delta Laver} and \ref{full nie zawierajacy Millera}.

\section{Category Case}
In this section we will focus on finding trees $T_1\se T_2\se \omega^{<\omega}$ of types mentioned in Introduction, satisfying $[T_1]\times [T_2]\se G$ for a given comeager set $G\se \baire\times \baire$. The main positive result is Theorem \ref{Mycieslki 1.5}. Theorem \ref{Miller razy Miller nie} and Propositions \ref{Silver razy Silver nie} and \ref{Miller nie uniformly perfect} show that the main result is somehow optimal.

Let $\QQ=\{q\in\baire: (\foralmostall n) (q(n)=0)\}$ be a set of rationals localized in $\baire$. On several occasions in this section we will use some specific countable dense subset of $\baire\times\baire$. Let us define it in the following way:
	\[
		Q=\{(p,q): p,q\in \QQ,\, \supp(p)=\supp(q)\tn{ and } p\neq q\},
	\]
where $\supp(q)=\max\{n\in\omega: q(n)\neq 0\}+1$. Since $\supp (q_1)=\supp (q_2)$ for every $q=(q_1, q_2)\in Q$, we may naturally extend the domain of $\supp$ to $\QQ\cup Q$ so that $\supp(q)=\supp(q_1)$.
\\
As a warm up let us consider a case of Laver trees.

\begin{proposition}\label{G delta Laver}
	There exists a dense $G_\delta$ set $G\se \omega^\omega$ such that $[L]\not\se G$ for every Laver tree $L.$
\end{proposition}
\begin{proof}
	Let $G=\{x\in\baire: (\existsinfty n\in\omega) (x(n)=0)\}$. Clearly, $G$ is $G_\delta$ and dense. Let $L$ be a Laver tree. Let $x\in [L]$ such that $x(n)\neq 0$ for every $n\geq|\stem(L)|$. Then $x\in [L]\bez G$.
\end{proof}
Let us notice that every nonempty open set is a body of a Laver tree.
\\
The following theorem shows that the perfect set in Mycielski Theorem cannot be replaced with a body of Miller tree.
\begin{theorem}\label{Miller razy Miller nie}
	There exists an open dense set $U\se \baire\times \baire$ such that $[T]\times [T]\not\se U \cup \Delta$ for every Miller tree $T$.
\end{theorem}
\begin{proof}
	Let $\{q^n: n\in \omega\}$ be an enumeration of $Q$ and let us set
	\[
		U=\bigcup_{k\in\omega}[q^k\rest (\supp (q^{k})+K(q^k))],
	\]
	where $K(q)=\max\{q_1(n), q_2(n): n\in\omega\}$ for $q=(q_1, q_2)\in Q$.
	\\
	Let $T$ be a Miller tree. Without loss of generality we may assume that for every $\sigma\in T$ either $|\succe_T(\sigma)|=1$ or $|\succe_T(\sigma)|=\omega$. We will pick points
	\begin{align*}
		x&={\sigma_0}^{\frown}{\sigma_1}^{\frown}...
		\\
		y&={\tau_0}^{\frown}{\tau_1}^{\frown}...
	\end{align*}
	from $[T]$ via induction. Let $\sigma_0=\tau_0=\stem(T)$. Let us assume the following notation
	\begin{align*}
		x_n&={\sigma_0}^{\frown}{\sigma_1}^{\frown}...\concat \sigma_n,
		\\
		y_n&={\tau_0}^{\frown}{\tau_1}^{\frown}...\concat\tau_n.
	\end{align*}
	Let us execute the step $n+1$. We set
	\[
		s_{n+1}=\min\{k\in\succe_T(x_n): k>|y_n|\}
	\]
	and $\sigma_{n+1}\es s_{n+1}$ such that ${x_n}\concat \sigma_{n+1}\in\splitt(T)$ and $|{x_n}\concat \sigma_{n+1}|>|y_n|$. In a similar fashion we proceed with $y_{n+1}$. We set
	\[
		t_{n+1}=\min\{k\in\succe_{T}(y_n): k>|x_{n+1}|\}
	\]
  and $\tau_{n+1}\es t_{n+1}$ such that ${y_n}\concat \tau_{n+1}\in\splitt(T)$ and $|{y_n}\concat \tau_{n+1}|>|x_{n+1}|$.
	\\
	We will show that
	\[
		(x,y)\in([T]\times[T])\bez (U\cup\Delta).
	\]
	It is clear that $(x,y)\in ([T]\times[T])\bez\Delta$. Let us suppose that $(x,y)\in U$. Then there is $q\in Q$ such that $(x,y)\in [q\rest(\supp(q)+K(q))]$. It follows that
	\begin{align*}
		q_1\rest(\supp(q)+K(q))&\se x,
		\\
		q_2\rest(\supp(q)+K(q))&\se y.
	\end{align*}
	Let us observe that since $q_1\neq q_2$, $|\stem(T)|<\supp(q)$. Let
	\[
		n=\max\{k\in\omega: {y_n}\concat t_{n+1}\se q_2\rest \supp(q)\}.
	\]
	In particular it means that $K(q)\geq t_{n+1}$ and $q_2\rest \supp(q)\sen {y_{n+1}}\concat t_{n+2}$. Let us also observe that
	\[
		|y_n|<|x_{n+1}|<|y_{n+1}|.
	\]
	It is the case that exactly one of the following holds:
	\begin{enumerate}
		\item $q_1\rest \supp(q)\sen {x_{n+1}}\concat s_{n+2}$;
		\item ${x_{n+1}}\concat s_{n+2}\se q_1\rest\supp(q)$.
	\end{enumerate}		
	If $(1)$ is true, then
	\[
		x\es(x_{n+1}\rest \supp(q))\concat \underbrace{ 0\concat ... \concat 0}_{K(q)},
	\]
	which gives a contradiction, since $K(q)\geq t_{n+1}>|x_{n+1}|$ and $s_{n+2}\neq 0$.
	\\
	If $(2)$ holds, then $K(q)\geq s_{n+2}$ and
	\[
		y\es(y_{n+1}\rest \supp(q))\concat \underbrace{ 0\concat ... \concat 0}_{K(q)},
	\]
	which is a contradiction because $K(q)\geq s_{n+2}>|y_{n+1}|$ and $t_{n+2}\neq 0$.
	\\
	Therefore $(x,y)\notin U$.
\end{proof}
Next result is concerned with replacing a perfect tree with a Silver tree. First, let us define some useful property of perfect trees. We will say that a perfect tree $T$ splits and rests, if
	\[
		(\forall \sigma,\tau\in T)(|\sigma|+1=|\tau|\land \sigma\se\tau \land \sigma \in \splitt(T)\IMP \tau\notin\splitt(T)).
	\]
\begin{lemma}\label{Silver splits and rests}
	For every Silver tree $T$ there exists a Silver tree $T'\se T$ that splits and rests.
\end{lemma}
\begin{proof}
	Let $n_0=\min\{|\sigma|: \sigma\in \splitt(T)\}$ and $s_0=\min\{n\in\omega: {\sigma_0}\concat n\in T\}$, where $\sigma_0\in T$ and $|\sigma_0|=n_0$. For $k>0$ let
	\begin{align*}
		n_k&=\min\{|\sigma|>n_{k-1}+1: \sigma\in\splitt(T)\},
		\\
		s_k&=\min\{n\in\omega: {\sigma_k}\concat n\in T\},
	\end{align*}
	where $\sigma_k\in T$ satisfies $|\sigma_k|=n_k$. Now, let
	\[
		B=\{x\in [T]: (\forall k\in\omega)(x(n_k)=s_k)\}
	\]
	and set
	\[
		T'=\{x\rest n: n\in\omega,\, x\in B\}.
	\]
	Then $T'$ is the desired tree.
\end{proof}
\begin{proposition}\label{Silver razy Silver nie}
	There exists an open dense set $U\se\baire\times\baire$ such that $[T]\times[T]\not\se U\cup\Delta$ for any Silver tree $T$.
\end{proposition}
\begin{proof}
	Let $Q=\{q^n: n\in\omega\}$ and set
	\[
		U=\bigcup_{n\in\omega}\big[\big(q^n_1\rest\big(\supp(q^n)\big)\big)\concat (0,0)\big]\times \big[\big(q^n_2\rest\big(\supp(q^n)\big)\big)\concat (1,1)\big].
	\]
	Let $T$ be a Silver tree. Without loss of generality we may assume that $T$ splits and rests (Lemma \ref{Silver splits and rests}). Let $(x,y)\in [T]\times[T]$, $x\neq y$, and suppose that $(x,y)\in U$. Then there is $q=(q_1, q_2)\in Q$ such that
	\begin{align*}
		{\big(q_1\rest\supp(q)\big)}\concat(0,0)\se x,
		\\
		{\big(q_2\rest\supp(q)\big)}\concat(1,1)\se y.
	\end{align*}
	Clearly
	\begin{align*}
		x(\supp(q)+1)&=0\neq 1=y(\supp(q)+1),
		\\
		x(\supp(q)+2)&=0\neq 1=y(\supp(q)+2),
	\end{align*}hence all of the nodes in $T$ of lengths $\supp$ and $\supp+1$ are splitting, which constitutes a contradiction with the splitting and resting property of $T$.
\end{proof}
The following lemmas are preparation to the main theorem of this section.
\begin{lemma}\label{lemat dla par otwartych}
	For every open dense set $U\se \baire\times\baire$ and two open sets $V_1, V_2\se \baire$ there are sequences $\sigma_1, \sigma_2\in \omega^{<\omega}$ satisfying $[\sigma_1]\se V_1$, $[\sigma_2]\se V_2$, $|\sigma_1|=|\sigma_2|$ such that $[\sigma_1]\times [\sigma_2]\se U$ and $[\sigma_2]\times [\sigma_1]\se U$.
\end{lemma}
\begin{proof}
	Let $U$, $V_1$ and $V_2$ be as in the formulation. $(V_1\times V_2)\cap U$ is open and nonempty, therefore there are sequences $\tau_1, \tau_2$ with $[\tau_1]\times [\tau_2]\se (V_1\times V_2)\cap U$. Repeating the argument, we find sequences $\tau_1'\es \tau_1$ and $\tau_2'\es \tau_2$ satisfying $[\tau_2']\times [\tau_1']\se U$. We may assume that $|\tau_1'|=|\tau_2'|$ (otherwise we extend the shorter one however we like). We set $\sigma_1=\tau_1'$ and $\sigma_2=\tau_2'$.
\end{proof}
For $\sigma, \tau\in\omega^{<\omega}$ and $U\se \baire\times\baire$ let us denote the fact that $[\sigma]\times [\tau]\se U$ and $[\tau]\times[\sigma]\se U$ by $\psi(\sigma, \tau, U)$. The following lemma is an extension of the previous one.
\begin{lemma}\label{lemat dla skonczonych ciagow}
	For every open dense set $U\se \baire\times\baire$, a finite sequence of open sets $(V_k: 0\leq k<n)$ in $\baire$ there is a sequence of sequences $(\sigma_k: 0\leq k<n)$ such that:
	\begin{enumerate}
		\item $[\sigma_k]\se V_k$ for all $0\leq k<n$,
		\item $|\sigma_k|=|\sigma_l|$ for all $0\leq k,l<n$,
		\item $\psi(\sigma^{k}_l, \sigma^{l}_k, U)$ for all distinct $0\leq k, l<n$.
	\end{enumerate}
\end{lemma}
\begin{proof}
	Let $U$ and $(V_k: k<n)$ be as in the formulation. Applying Lemma \ref{lemat dla par otwartych} multiple times we will construct inductively a sequence $(\sigma_l^k: k,l<n;\, k\neq l)$ of sequences satisfying:
	\begin{enumerate}
		\item $(\forall k,l<n)(l\neq k\IMP [\sigma^k_l]\se V_k)$;
		\item $(\forall k)([\sigma^k_0]\es [\sigma^k_1]\es ... \es [\sigma^k_{k-1}]\es [\sigma^k_{k+1}]\es ... \es [\sigma^k_{n-1}])$;
		\item $\psi(\sigma^{k}_l, \sigma^{l}_k, U)$ holds for all distinct $k, l<n$.
	\end{enumerate}
	At the step $0$ first we find $\sigma^0_1$ and $\sigma^1_0$ such that $[\sigma^0_1]\se V_0$, $[\sigma^1_0]\se V_1$ and $\psi(\sigma^0_1, \sigma^1_0, U)$.
	Then for every $k<n, k\neq 0,1,$ we choose $\sigma^0_k$ and $\sigma^k_0$ satisfying $[\sigma^0_k]\se [\sigma^0_{k-1}], [\sigma^k_0]\se V_k$, and $\psi(\sigma^0_k,\sigma^k_0,U)$.
	\\
	Let us execute the step $k$, $0<k<n$. We pick $\sigma^k_{k+1}$ and $\sigma^{k+1}_k$ satisfying $[\sigma^k_{k+1}]\se [\sigma^k_{k-1}]$, $[\sigma^{k+1}_k]\se [\sigma^{k+1}_{k-1}]$ and $\psi(\sigma^k_{k+1},\sigma^{k+1}_k, U)$. For $l>k+1$ we find $\sigma^k_{l}$ and $\sigma^l_{k}$ such that $[\sigma^k_{l}]\se [\sigma^k_{l-1}]$, $[\sigma^l_{k}]\se [\sigma^l_{k-1}]$ and $\psi(\sigma^k_{l}\sigma^l_{k}, U)$. The construction is completed.
	\\
	Let us set $\sigma_k'=\sigma^k_{n-1}$ for every $k<n$. If lengths of $\sigma_k', 0\leq k<n$ are the same, then we set $\sigma_k=\sigma_k'$ for each $0\leq k<n$. If not, let $N=\max\{|\sigma_k'|: 0\leq k<n\}$, and let us set
	\[
		\sigma_k={\sigma_k'}\concat\underbrace{ 0 \concat ...\concat 0}_{N-|\sigma_k'|}
	\]
	for each $0\leq k<n$. Then lengths of these sequences match and properties established during the construction are not compromised.
\end{proof}

\begin{theorem}\label{Mycieslki 1.5}
For every comeager set $G$ of $\w^\w \times \w^\w$ there exists a Miller tree $M\subseteq \w^{<\w}$ and a uniformly perfect tree $P\se M$ such that $[P]\times [M] \subseteq G\cup \Delta.$
\end{theorem}
\begin{proof} Let us assume that $G=\bigcap_{n\in\w} U_n$ where $(U_n)_{n\in\w}$ is a descending sequence of open dense subsets of $\w^\w \times \w^\w$. We will construct recursively a sequence $(B_n: n\in\omega)$ of sets. $B_n=\{\tau_\sigma: \sigma\in n^{\leq n}\}$ should consist of nodes satisfying:
\begin{enumerate}
	\item $\tau_\0=\0$, $\tau_{\sigma_1}\se \tau_{\sigma_2}$ for $\sigma_1\se\sigma_2$ and ${\tau_\sigma}\concat k\se\tau_{\sigma\concat k}$;
	\item $\tau_{\sigma^\frown k} \cap \tau_{\sigma^\frown j}=\tau_\sigma$ for $\sigma\in n^{<n}$ and distinct $k, j< n$;
	\item for $n>0$ and all $\tau, \tau'\in B_n\bez B_{n-1}$ $\psi(\tau, \tau', U_n)$ holds;
	\item If $\sigma_1, \sigma_2\in \{0,1\}^{\leq n}$ then $|\tau_{\sigma_1}|=|\tau_{\sigma_2}|$.
\end{enumerate}
At the step $0$ we set $\tau_\0=\0$ and $B_0=\{\tau_\0\}$. Next, we set $\tau_0, \tau_1\es\tau_\0$ so that $\psi(\tau_0, \tau_1, U_2)$ (Lemma \ref{lemat dla par otwartych}), and $\tau_{00}, \tau_{01}\es\tau_0, \tau_{10}, \tau_{11}\es\tau_1$ with accordance to Lemma \ref{lemat dla skonczonych ciagow}. We set
    \[
      B_1=B_0\cup\{\tau_0, \tau_1\} \tn{ and }  B_2=B_1\cup\{\tau_{00}, \tau_{01}, \tau_{10}, \tau_{11}\}.
    \]
Now, let us assume that we already have a set $B_n$ with the above properties and let us execute the step $n+1$, $n>1$. First we set $\tau_{\sigma\concat n}$ for $\sigma\in n^{<n}$ and $\tau_{\sigma\concat k}$, $\sigma\in n^n$, $k<n+1$, in such a way that they have the same lengths, propagate the condition $(1)$ and $(2)$, and $\psi(\tau_{\sigma_1}, \tau_{\sigma_2}, U_{n+1})$ for all distinct
	\[
		\sigma_1, \sigma_2 \in\{\sigma\concat n: \sigma\in n^{<n}\}\cup\{\sigma\concat k: \sigma\in n^n, \, k<n+1\}.
	\]
	Next, we set $\tau_{\sigma\concat k}$ for $\sigma\in (n+1)^{<n+1}\bez n^{\leq n}$ and $k<n+1$ in a similar fashion.
	\\
	This completes the construction. Let us set $B=\bigcup_{n\in\omega}B_n$ and
\begin{align*}
	M&=\{\tau\in\omega^{<\omega}: (\exists \tau'\in B) (\tau\se\tau')\},
	\\
	P&=\{\tau\in\omega^{<\omega}: (\exists n\in\omega)(\exists \sigma\in 2^{n})(\tau\se \tau_{\sigma})\}.
\end{align*}
Clearly, $M$ is a Miller tree. Furthermore, $P\se M$ is a uniformly perfect tree thanks to the condition $(4)$. We will show that $[P]\times [M]\se G\cup\Delta$. Let $(x,y)\in [P]\times[T]$, $x\neq y$. We claim that there exists $\alpha\in\baire$ such that
	\[
		(\forall n\in\omega) (\tau_{\alpha\rest n}\se y).
	\]
We will define $\alpha=(a_0, a_1, ...)$ via induction. Let us observe that $y\rest 1\se \tau_{y(0)}$ and $\tau_{y(0)}$ is the shortest sequence from $B$ possessing such a property. Therefore, $y\rest |\tau_{y(0)}|=\tau_{y(0)}$, otherwise there would be $\tau\in B$ such that $\tau_{y(0)}\sen\tau$ and $y\rest |\tau_{y(0)}|\se \tau$, which is a contradiction. We set $a_0=y(0)$.
\\
Next, let us assume that we already have a strictly ascending sequence $(a_{k}: k<n)$ of natural numbers with a property $\tau_{a_0a_1...a_{n-1}}\se y$ for every $k<n$. As previously, we see that
	\[
		y\rest (|\tau_{a_0a_1...a_{n-1}}|+1)\se \tau_{a_0a_1...a_{n-1}y(|\tau_{a_0a_1...a_{n-1}}|)}
	\]
and that $\tau_{a_0a_1...a_{n-1}y(|\tau_{a_0a_1...a_{n-1}}|)}$ is the shortest sequence from $B$ with such a property. Hence $\tau_{a_0a_1...a_{n-1}y(|\tau_{a_0a_1...a_{n-1}}|)}\se y$, so we set $a_n=y(|\tau_{a_0a_1...a_{n-1}}|)$. This completes the definition of $\alpha$.
\\
Now, let us fix $N\in\omega$. There exists $N'\geq N$ such that $\tau_{\alpha\rest N}\in B_{N'}\bez B_{N'-1}$. Furthermore, there exists $\sigma\in 2^{N'}$ such that $\tau_\sigma\se x$. Then $[\tau_\sigma]\times [\tau_{\alpha\rest N}]\se U_{N'}\se U_N$, hence also $(x,y)\in U_N$. $N$ was chosen arbitrarily, thus $(x,y)\in G$.
\end{proof}
Let us make some remarks. The Miller tree $T$ in the above theorem has a nice property. For each $\tau\in T$ the set $\succe_T(\sigma)=\omega$ or $|\succe_T(\sigma)|=1$. Let us also observe that one cannot make this Miller tree uniformly perfect.
\begin{proposition}\label{Miller nie uniformly perfect}
	There exists a $G_\delta$ set $G$ such that $[T]\not\se G$ for every uniformly perfect Miller tree.
\end{proposition}
\begin{proof}
	For every $n\in\omega$ let $G_n=\bigcup_{q\in\QQ}[q\rest(\supp(q)+K(q)+n))]$. Let $T$ be a uniformly perfect Miller tree. Without loss of generality we may assume that for every $\sigma\in T$ we have $|\succe_T(\sigma)|\in\{1, \omega\}$. Let $\{n_k: k\in\omega\}$ be an enumeration of
	\[
            \{n\in\omega: \omega^n\cap T\se\splitt(T)\}
	\]
in an ascending order. We find $x\in[T]$ such that $x(n_{k})>n_{k+1}$ for each $k\geq 0$. Let $N>n_0$ and let us suppose that $x\in G_N$. Then there exists $q\in\QQ$ such that $q\rest(\supp(q+K(q)+N))\se x$. If $\supp(q)<n_0$, then $x(n_0)=0$, a contradiction. Let us assume that $\supp(q)\geq n_0$ then, and let
	\[
		m=\min\{k\in\omega: \supp(q)<n_k\}.
	\]
	Let us notice that $m>0$. $n_{m-1}\leq \supp(q)$, hence $K(q)\geq x(n_{m-1})>n_m$, which implies that $x(n_m)=0$. A contradiction, thus the proof is complete.
\end{proof}
Let us observe that each nonempty open set contains a body of uniformly perfect Miller tree, e.g. a basic clopen set.

\section{Measure Case}

This section is devoted to possible enhancements of two-dimensional Mycieski theorem for the measure case. 
Proposition \ref{full nie zawierajacy Millera} mirrors Proposition \ref{G delta Laver}. It shows that we may exclude Miller trees from further considerations. Hence, in consecutive results we work in the Cantor space. The main theorem of this section (Theorem \ref{fulluniformly}) shows that we can inscribe the square of a body of uniformly perfect tree into a set of measure one (modulo diagonal). Proposition \ref{fullSilver} shows that it is not true in the case of Silver trees and Proposition \ref{full contains Silver} shows that no one-dimensional counterexample is feasible.

\begin{proposition}\label{full nie zawierajacy Millera}
	Let $\mu$ be a strictly positive probabilistic measure on $\omega^\omega$. Then there exists an $F_{\sigma}$ set $F$ of measure $1$ such that $[T]\not\se F$ for every Miller tree $T$.
\end{proposition}
\begin{proof}
	Let $\eps_n=\frac{1}{2^n}$ for every $n>0$. We will construct inductively a sequence $(F_n: n\in\omega)$ of closed subsets of $\baire$. For every $n\in\omega$ let us set
	\begin{align*}
		m^n_1&=\min\{n\in\omega: \sum_{i=0}^{n}\mu([i])>1-\eps_{n+1}\},
		\\
		T^n_1&=\{\0, (i): i\leq m_1^n\},
	\end{align*}
and for $k>1$ let
	\begin{align*}
		m^n_k&=\min\{j\in\omega: (\forall \sigma\in \omega^{k-1}\cap T^n_{k-1}) (\sum_{i=0}^{j}\mu([\sigma\concat i])>(1-\eps_{k+n})\mu([\sigma]))\},
		\\
		T^n_k&=T^n_{k-1}\cup\{\sigma\concat i: i\leq m^n_k, \, \sigma\in T^n_{k-1}\}.
	\end{align*}
	Then we set $T_n=\bigcup_{i\in\omega}T^n_i$ and $F_n=[T_n]$. Finally, let $F=\bigcup_{n\in\omega}F_n$. To see that $F$ is the desired set, let us approximate its measure. For each $n\in\omega$ we have
	\[
		\mu(F)\geq\mu(F_n)>\prod_{i=1}^{\infty}(1-\eps_{n+i})=\prod_{i=1}^{\infty}(1-\frac{1}{2^{n+i}})\to^{n\to\infty} 1.
	\]
\end{proof}

From now on we will work in $2^\omega$ exclusively. By $\lambda$ we will denote standard product measure on $2^\omega$. We will use the same notation for standard product measure on $2^\omega\times 2^\omega$.
\\
Let $\sigma\in 2^{<\omega}$. For a given set $A\se [\sigma]^2$ let us denote $A^s=A\cap A^{-1}$, where $A^{-1}=\{(x,y): (y,x)\in A\}$. Let us observe that if $\lambda(A)=(1-\eps)\lambda([\sigma])^2$, then $\lambda(A^s)=(1-2\eps)\lambda([\sigma])^2$. For every set $B$ we will denote a set of its density points by $B^*$.
\begin{theorem}\label{fulluniformly}
	Let $F$ be a subset of $2^\omega\times 2^\omega$ of full measure. Then there exists a uniformly perfect tree $T\se2^{<\omega}$ satisfying $[T]\times [T] \se F\cup \Delta$.
\end{theorem}
\begin{proof}
	Let $F\se 2^\omega\times 2^\omega$ be a set of full measure and let us assume that $F=\bigcup_{n\in\omega}F_n$, where $(F_n:n\in\omega)$ is an ascending sequence of closed sets. Let us fix a sequence $\varepsilon_n=\frac{1}{2^{2n+3}}$, $n\in\omega$.
We shall construct inductively:
\begin{itemize}
	\item a collection of clopen sets $\{[\tau_{\sigma}]: \sigma\in 2^{<\omega}\}$;
	\item two sequences of natural numbers $(k_{n}: n\in\omega)$ and $(N_n: n\in\omega\bez\{0\})$;
	\item a sequence of pairs $((x_n,y_n): n\in\omega\bez\{0\})$ from $2^\omega\times 2^\omega$;
	\item a collection of points $\{t_{\sigma}: \sigma\in 2^{<\omega}\}$ from $2^\omega$;
	\item a sequence $(B_n: n\in\omega)$ of subsets of $2^\omega\times 2^\omega$;
\end{itemize}
satisfying the following conditions for all $\sigma, \eta\in 2^{<\omega}$ and $n\in\omega$:
\begin{enumerate}
	\item $\tau_\sigma\se \tau_\eta\IFF\sigma\se\eta$;
	\item $|\sigma|=|\eta|\IMP[\tau_\sigma] \times [\tau_\eta]\cap F_{k_{|\sigma\cap\eta|}}\neq\0$;
	\item $|\sigma|=|\eta|\IMP|\tau_\sigma|=|\tau_\eta|;$
	\item The set
	\[
		B_n=\bigcap_{\sigma, \eta \in \{0,1\}^{n}}\Big(\big(([\tau_\sigma]\times [\tau_\eta])\cap F_{k_{|\sigma\cap\eta|}}\big)-(t_\sigma, t_\eta)\Big)^s
	\]
	has a positive measure.
\end{enumerate}
	Let $\tau_\0=\0, t_\0=0\concat 0\concat ...,$ and set
	\begin{align*}
		k_0&=\min\{k\in\omega: \lambda(F_k)>1-\eps_0\},
		\\
		B_0&=F_{k_0}^s.	
	\end{align*}
	Next, let $(x_1,y_1)\in B_0^*$, with a requirement $x_1\neq y_1$, and set
	\[
		N_1=\min\{N\in\omega: \lambda\big([x_1\rest N]\times [y_1\rest N])\cap B_0\big)>\frac{1}{2^{2N}}(1-\eps_1)\;\land\;x_1\rest N\neq y_1\rest N\}.
	\]
	Then set
	\begin{align*}
		\tau_0&=x_1\rest N_1,
		\\
		\tau_1&=y_1\rest N_1,
	\end{align*}
	and let $t_i\in 2^{\omega}$ be such that $t_i\rest N_1=\tau_i$ and $t_i(n)=0$ for $n\geq N_1$, $i\in\{0,1\}$. Also, set
	\[
		k_1=\min\{k>k_0: (\forall \sigma, \eta\in\{0,1\})\Big(\lambda\big(\big(([\tau_\sigma]\times [\tau_\eta])\cap F_{k}\big)-(t_\sigma, t_\eta)\big)\big)>\frac{1}{2^{2N_1}}(1-\eps_1)\Big)\}
	\]	
	and
	\[
		B_1=\bigcap_{\sigma, \eta\in\{0,1\}}\Big(\big(([\tau_\sigma]\times [\tau_\eta])\cap F_{k_{|\sigma\cap\eta|}}\big)-(t_\sigma, t_\eta)\Big)^s
	\]
	Let us execute the step $n+1$, $n>0$. We pick $(x_{n+1}, y_{n+1})\in B_n^*$, $x_{n+1}\neq y_{n+1}$, and set
	\begin{align*}
		N_{n+1}=\min\{N>0:\; &\lambda\big(([x_{n+1}\rest N]\times [y_{n+1}\rest N])\cap B_n\big)>\frac{1}{2^{2N}}(1-\eps_{n+1})\;\land
		\\
		&\land \; x_{n+1}\rest N\neq y_{n+1}\rest N\}.
	\end{align*}
	Then for every $\sigma\in \{0,1\}^n$ let
	\begin{align*}
		\tau_{\sigma\concat 0}&=x_{n+1}\rest N_{n+1}+t_\sigma\rest N_{n+1},
		\\
		\tau_{\sigma\concat 1}&=y_{n+1}\rest N_{n+1}+t_\sigma\rest N_{n+1},
	\end{align*}
	and for $i\in\{0,1\}$ let
	\begin{align*}
		t_{\sigma\concat i}={\tau_{\sigma\concat i}}\concat 0\concat 0 \concat ...\,.
	\end{align*}
	Let us set
	\begin{align*}
		k_{n+1}=\min\{k>k_n:\; & (\forall \sigma, \tau\in\{0,1\}^{n+1})
		\\
		& \Big(\lambda\big(\big(([\tau_\sigma]\times [\tau_\eta])\cap F_{k}\big)-(t_\sigma, t_\eta)\big)\big)>\frac{1}{2^{2N_{n+1}}}(1-\eps_{n+1})\Big).
	\end{align*}
	Finally, let us set
	\[
		B_{n+1}=\bigcap_{\sigma, \eta \in \{0,1\}^{n+1}}\Big(\big(([\tau_\sigma]\times [\tau_\eta])\cap F_{k_{|\sigma\cap\eta|}}\big)-(t_\sigma, t_\eta)\Big)^s.
	\]
	Since
	\[
		\lambda(B_{n+1})>\frac{1}{2^{2N_{n+1}}}(1-2^{2n+2}\eps_{n+1})>0,
	\]
	we may carry on with the construction, thus it is complete. A set
	\[
		T=\{\tau\in 2^{<\omega}: (\exists \sigma\in 2^{<\omega})(\tau\se\tau_\sigma)\}
	\]
	is the uniformly perfect tree we were looking for.
\end{proof}

Now, let us recall the notion of small sets (see \cite{BartJud}) connected to null subsets of $2^\omega$.
\begin{definition}
	$A\se 2^\omega$ is a small set if there is a partition $\mathcal{A}$ of $\omega$ into finite sets and a collection $(J_a)_{a\in\mathcal{A}}$ such that $J_a\subseteq 2^a$, $\sum_{a\in\mathcal{A}}\frac{|J_a|}{2^{|a|}}<\infty$ and
	$$
	A=\{x\in 2^\omega:\ (\exists^\infty a\in\mathcal{A})( x\rest a\in J_a)\}.
	$$
\end{definition}

Let us remark that each small set is a null set. Moreover, every null set is a union of two small sets (see \cite{BartJud}).


The space $2^\omega\times 2^\omega$ is canonically homeomorphic to $2^\omega$, so it is natural to consider a notion of small set in $2^\omega\times 2^\omega$.

\begin{proposition}\label{fullSilver}
	There exist a small set $A\se 2^\omega\times 2^\omega$ such that $(A\cap [S]\times [S])\setminus \Delta \neq\emptyset$ for any Silver tree $S\se 2^{<\omega}.$
\end{proposition}	
\begin{proof}
	Let $\{I_n\}_{n\in\omega}$ be a partition of $\omega$ into finite segments such that $|I_n|\ge n$.\\
	Clearly, $\{I_n\times I_m\}_{n,m\in\omega}$ forms a partition of $\omega\times\omega$.
	Define
	$$J_{n,m}=\left\{\begin{array}{l@{\;\text{ if }\;}l}
	\emptyset & n\neq m\\
	\{(x,x):\ x\in 2^{I_n}\} & n=m
	\end{array} \right.$$
	Notice that $\displaystyle\frac{|J_{n,n}|}{2^{|I_n\times I_n|}}=\frac{1}{2^{|I_n|}}\le\frac{1}{2^n}.$ So
	$$
	A=\{(x,y)\in 2^\omega\times 2^\omega:\ (\exists^\infty n\in\omega)( x\rest I_n = y\rest I_n)\}
	$$
	is a small set.
	Let $S$ be a Silver tree. Let $x,y\in [S]$ be such that $(\forall^\infty k)( x(k)=y(k))$, but $x\neq y$.
	Clearly, $(x,y)\in A\setminus\Delta$.
\end{proof}

\begin{proposition}\label{full contains Silver}
	Every closed subset of $2^\omega$ of positive Lebesgue measure contains a Silver tree.
\end{proposition}
\begin{proof}
	Let $F\se2^\omega$ be a closed set of positive measure. Let $\eps_n=\frac{1}{2^{n+3}}$ for every $n\in\omega$. Let $x_0$ be a density point of $F$ and let $\sigma_0\se x_0$ such that $\lambda([\sigma_0]\cap F)>(1-\eps_0)\lambda([\sigma_0])$. Since
	\[
		\lambda([\sigma_0]\cap F)=\lambda([{\sigma_0}\concat 0]\cap F)+\lambda([{\sigma_0}\concat 1]\cap F),
	\]
	we have
	\[
		\lambda([{\sigma_0}\concat i]\cap F)>(\frac{1}{2}-\eps_0)\lambda([\sigma_0]),\, i\in\{0,1\}.
	\]
	Let $t_1=(\underbrace{0,0,...,0}_{|\sigma_0|},1,0,...)$ and let us observe that $([{\sigma_0}\concat 1]\cap F)-t_1\se [{\sigma_0}\concat 0]$. Since
	\[
		\lambda([{\sigma_0}\concat i])=\frac{1}{2}\lambda([\sigma_0])
	\]
	for $i\in\{0,1\}$, we have
	\[
		\lambda\big([{\sigma_0}\concat 0] \cap F \cap([{\sigma_0}\concat 1]\cap F)-t_1)\big)\geq (\frac{1}{2}-2\eps_0)\lambda([\sigma_0])>0.
	\]
	Let us assume that at the step $n+1$ we have a sequence $(\sigma_k: k\leq n)$ of finite $0$-$1$ sequences. Let $0_k=(\underbrace{0,0,...,0}_{k})$ and for every $s\in 2^{n+1}$ let us denote
	\begin{align*}
		\tau_s&={\sigma_0}\concat s(0)\concat {\sigma_1}\concat s(1)\concat ... \concat {\sigma_n} \concat s(n),
		\\
		t_s&={0_{|\sigma_0|}}\concat s(0)\concat {0_{|\sigma_1|}}\concat s(1)\concat ... {0_{|\sigma_n|}}\concat s(n)\concat 0 \concat...\, ,
	\end{align*}
	and assume that a set
	\[
		B_n=\bigcap_{s\in 2^{n+1}}\big(([\tau_s]\cap F)-t_s)\big)
	\]
	has a positive measure.	Let $x_{n+1}\in B_n^*$. Then
	\begin{align*}
		&x_{n+1}+t_s\in([\tau_s]\cap F)^*,
	\end{align*}
	for every $s\in 2^{n+1}$. Let us observe that for a given sequence $\sigma\in 2^{<\omega}$ satisfying ${\tau_{0_{n+1}}}\concat \sigma\se x_{n+1}$ it is also true that ${\tau_{s}}\concat \sigma\se x_{n+1}+t_s$ for every $s\in 2^{n+1}$. Hence, we may pick $\sigma_{n+1}$ such that for every $s\in 2^{n+1}$
	\[
		\lambda(([{\tau_s}\concat\sigma_{n+1}]\cap F)>(1-\eps_{n+1})\lambda([{\tau_{0_{n+1}}}\concat\sigma_{n+1}]).
	\]
	Similarly to the first step, we see that for every $s\in 2^{n+1}$
	\[
		\lambda\big(([{\tau_s}\concat{\sigma_{n+1}}\concat i]\cap F\big)\geq (\frac{1}{2}-\eps_{n+1})\lambda([{\tau_s}\concat{\sigma_{n+1}}])=(1-2\eps_{n+1})\lambda([\tau_{0_{n+2}}])
	\]
	and eventually
	\[
		\lambda (B_{n+1})=\lambda \Big(\bigcap_{s\in 2^{n+2}}\big(([\tau_s]\cap F)-t_s\big)\Big)>(1-2^{n+3}\eps_{n+1})\lambda([{\tau_{0_{n+2}}}])=(1-\frac{2^{n+3}}{2^{n+4}})\lambda([{\tau_{0_{n+2}}}])>0.
	\]
	This allows us to carry on with the construction, thus it is complete. Then
	\[
		T=\{\sigma\in 2^{<\omega}: (\exists s\in 2^{<\omega})(\sigma\se\tau_s)\}
	\]
	is the desired Silver tree.
\end{proof}

\section{Nonstandard proofs}
In this section we prove a result concerned with implications of adding a Cohen real. As a consequence we obtain a nonstandard proof of strengthened two-dimensional version of Mycielski Theorem (see \cite{Mycielski}). We use Shoenfield Absoluteness Theorem. Using similar methods we prove a strengthened Egglestone Theorem (see \cite{Eggleston}).

By canonical Polish spaces we understand countable products of $\baire, 2^\w, [0,1],\bbr$ and $\Perf(\RR)$ - a space of perfect subsets of $\RR$. We say that $\varphi$ is $\Sigma_2^1$-sentence if for some canonical Polish spaces $X,Y$ and Borel set $B\subseteq X\times Y$ $\varphi$ is of the form:
$$
(\exists x\in X)(\forall y\in Y) (x,y) \in B.
$$
The Borel set $B$ has its so called Borel code $b\in\baire$ (see \cite{Kech}). The triple $(X, Y, b)$ is a parameter of our $\Sigma^1_2$-sentence $\varphi$.
Now, let us recall Schoenfield Absoluteness Theorem.
\begin{theorem}[Schoenfield]\label{Schoenfield} Let $M\subseteq N$ be standard transitive models of ZFC and ${\w_1^N \subseteq M.}$ Let $\varphi$ be a $\Sigma_2^1$-sentence with a parameter from model $M.$ Then
\[
	M\models \varphi \IFF N\models \varphi.
\]
\end{theorem}
Let us observe that if $N$ is a generic extension of standard transitive model $M$ of ZFC then $Ord^M = Ord^N$ and $\w_1^N\subseteq M.$

A method of providing nonstandard proofs of mentioned theorems will be as follows. We start with a standard transitive model $M$ of ZFC and find a generic extension $N$ of $M$ in which the theorem can be easily proved. Then we verify that the theorem forms a $\Sigma_2^1$-sentence. We apply Schoenfield Absoluteness Theorem to deduce that it is true in the ground universe $M.$

Before we proceed let us introduce some additional notation. For a tree $T\se\w^{<\w}$ we define
	\[
		\tips(T) = \{ \sigma\in T:\; \neg(\exists \tau\in q)\; (\sigma\subseteq \tau \land \sigma\ne \tau) \}.
	\]
	Let us recall that for a tree $T\se\w^\w$ and a node $\sigma\in T$ we define
	\[
		\rank_T(\sigma)=\sup\{\rank_T(\tau)+1: \tau\in T \;\land\; \sigma\sen\tau\}.
	\]
We will denote a height of a given tree $T$ by $\tn{ht}(T)=\rank_T(\0)$.
\\
We say that a tree $T\se \w^{<\w}$ is
	\begin{itemize}
		\item evenly cut if 
		and there is $n\in\w$ such that $\tips(q)\subseteq \w^n$ and $\tn{ht}(q) = n$;
		\item a slalom tree if
		\[
			(\forall \sigma\in \w^{<\w}) (\exists I\se\omega) \big(I \tn{ is an interval } \land (\forall \tau\in T)(I\se |\tau|\IMP \tau\rest I= \sigma)\big).
		\]
	\end{itemize}
	Let observe that the definition of slalom trees is arithmetic and so it is absolute between transitive  models of ZFC. We will say that a set $P\subseteq \w^\w$ is slalom perfect if it is a body of a perfect slalom tree. Let us notice that for every slalom perfect set $P$ and every $\sigma\in\w^{<\w}$ there is an interval $I\subseteq \w$ such that for every $x\in P$ we have $x\restricted I = \sigma.$

\begin{theorem}\label{cohen_real_extension} After adding one Cohen real
	there is a perfect slalom tree $T$ such that ${[T]\times [T] \subseteq W\cup \Delta}$ for every dense $G_\delta$ set $W \subseteq \w^\w \times \w^\w$ from the ground model. 
\end{theorem}
\begin{proof} Let $V$ be a ground model of ZFC. We will show that after adding one Cohen real to $V$ there is a perfect tree $T$ such that ${[T]\times [T] \subseteq U \cup \Delta}$ for every open dense set $U\se \w^\w\times\w^w$ coded in $V$. Let us define a poset $(\cc,\le)$ as follows:
	\[
		\cc = \{ p\se \w^{<\w}:\; p \text{ is an evenly cut and finite tree}\},
	\]
	and for every $p,q\in \cc$
	\[
		p\le q \tn{ ($p$ is stronger than $q$)} \IFF q\subseteq p \;\land\; p\cap \w^{\tn{ht}(q)} = \tips (q).
	\]	
Clearly, $(\cc, \leq)$ is a forcing adding one Cohen real. Let $G\subseteq \cc$ be any $\cc$-generic filter over $V.$ In $V[G]$ let us define a generic set
${T_G = \bigcup G}$. 
We have the following
\begin{claim*}\label{generic_claim} The following statements are true:
\begin{enumerate}
    \item $T_G$ is a slalom perfect tree.
    \item For any open dense set $U\subseteq \w^\w$ in $V$ and natural $n$ a set
    $$
        D_{n,U}=\{p\in\cc:\; (\forall t,s\in\tips(p)) (n\le |t|,|s|\land (s\ne t) \IMP [t]\times [s] \subseteq U \}
    $$
    is dense in $(\cc,\le)$.
    \item Fix a name $\dot{x}\in V^\cc$ and $p,q\in G$. Assume that 
    \[
        p\force \dot{x}\in [T_G] \ti{ and } q\force \dot{x}\restricted n \subseteq s \ti{ for some } n\in\w \ti{ and } s\in q.
    \]
     Then there exists $r\in G$ and $m\ge n$ such that $r\le p,q$ and $r\force \dot{x}\restricted m \in \tips(q).$
\end{enumerate}
\end{claim*}
\begin{proof}[Proof of the Claim] (1) follows from the density argument. That is, to see that $T_G$ is a perfect tree let us observe that for every $p\in \cc$ and every $t\in p$ the set
\[
D_{p,t}=\{r\in\cc:\; (\exists s\in r)(\exists m,m'\in\w)(t\subseteq s \;\land\; m\ne m'\;\land\; s^\with m,s^\with m'\in r)\}
\]
is defined in $V$ and it is dense below $p.$ To prove that $T_G$ is a slalom tree it is enough to observe that for every $s\in\w^{<\w}$ the following ground model set
\[
E_s=\{ p\in\cc:\; (\exists I\se\w)\big(I\text{ is an interval}\;\land\; |I|=|s|\land (\forall t\in\tips(p))(t\restricted I = s)\big) \}
\]
is dense in $\cc$.
\\
(2) follows directly from Lemma \ref{lemat dla skonczonych ciagow}.
\\
To show (3) let $n\in\w$, $p,q\in G$ and $\dot{x}\in V^\cc$ such that ${p\force \dot{x}\in \dot{[T_G]}}$ and ${p\force \dot{x}\restricted n \in q.}$ Let $m_0=\tn{ht}(q)$. 
Then there exists $r',q'\in G$ such that ${r'\force \dot{x}\restricted m_0\in q'.}$ $G$ is a filter, hence there exists $r\in G$ such that $r\le q,q',r',p$ and $r\force \dot{x}\restricted{m_0}\in q'$, so ${r\force \dot{x}\restricted{m_0}\in r}$. Let us observe that $r\le q$ and $\dot{x}_G\restricted m_0\notin q.$ Then there is $t\in\tips(q)$ such that ${r\force t\subseteq \dot{x}_G\restricted m_0\subseteq \dot{x}_G.}$
\end{proof}
Now let $\dot{x},\dot{y}\in V^\cc$ and $p\in G$, $k\in\w$ be such that ${p\force \dot{x},\dot{y}\in\dot{[T_G]}}$ and ${\dot{x}\restricted{k} \ne \dot{y}\restricted{k}.}$ Then there are $p_1,q_x,q_y\in G$ such that $p_1\le p$ and for some $n_x,n_y> k$ we have ${p_1\force \dot{x}\restriction n_x\in q_x\land \dot{y}\restriction n_y\in q_x}$. Since $G$ is a filter in $(\cc,\le)$, there exist $p',q\in G$ such that $p'\le p_1,$ $q\le q_x,q_y$ and $p'\force \dot{x}\restriction n_x,\dot{y}\restriction n_y\in q.$ By (1) of Claim 
there is a condition $q'\in G$ such that $q'\le q$ and for every $t,s\in \tips(q')$ if $t\ne s$ then $[t]\times [s] \subseteq U.$ By (2) of Claim we can find a generic condition $r\in G$ such that $r\le p'$ and there are $s,t\in\tips(q')$ such that $r\force \dot{x}\restriction m_x,\dot{y}\restriction=s\land m_y =t$ for some $m_x\ge n_x\ge k$ and $m_y\ge n_y\ge k.$ Then for the $r\in G$ we have
\[
	r\force (\dot{x},\dot{y})\in [\dot{x}\restriction m_x]\times [\dot{y}\restriction m_y] \in \hat{U}.
\]
The dense open set $U$ from the ground model was chosen arbitrarily, hence $[T_G]\times [T_G] \subseteq W \cup \Delta$ for any dense $G_\delta$ planar set $W$ from $V[G]$.
\end{proof}

\begin{theorem}
For every $G\in G_\delta$ dense set in $ \w^\w \times \w^\w$ there exists a slalom perfect set $P\subseteq \w^\w$ satisfying	$P\times P \subseteq G\cup \Delta.$ 
\end{theorem}
\begin{proof} Now let $V$ be a transitive model of ZFC and $W\in V$ be a $G_\delta$ dense set in $\w^\w\times \w^\w.$ Let $G\subseteq \cc$-generic filter over $V.$ Then by Theorem \ref{cohen_real_extension} there is a perfect tree $T_G$ in $V[G]$ such that $[T_G]\times[T_G]\subseteq W \cup \Delta.$ Here $W\in V$, hence the formula
	\[
		(\exists P\in \Perf(\w^\w))(\forall x,y\in P)\; (x\ne y \then (x,y)\in W)
	\]
	is $\Sigma^1_2$-sentence with a parameter from $V$. By Shoenfield Absoluteness Theorem the above formula holds in $V.$
\end{proof}
Our next result is concerned with a generalization of Egglestone Theorem. In \cite{Zeberski} such a generalization was proved using Shoenfield Absoluteness Theorem.
We will give yet another proof of this result.
In \cite{Zeberski} the author worked with a generic extension in which $\cof(\Ii)=\w_1<\c$, $\Ii\in\{\cm,\cn\}$ ($\Mm$ and $\Nn$ denote ideals of meager and null sets respectively). Our proof is based on a generic extension in which $\w_2 <\add(I)\le\c.$
\\
Let us recall that for ideals $\Ii\se P(X), \Jj\se P(Y)$ we define a Fubini product $\Ii\otimes\Jj$ of these ideals in the following way
	\[
		A\in \Ii\otimes\Jj \IFF (\exists B\in \tn{Bor}(X\times Y))(A\se B\;\land\; \{x\in X: B_x\notin\Jj\}\in\Ii),
	\]
	where $\Bor(X\times Y)$ is a family of Borel subsets of $X\times Y$ and $B_x=\{y\in Y: (x,y)\in B\}$ is a vertical section of the set $B$ (similarly we define a horizontal section $B^y$).
We say that the pair $(\Ii, \Jj)$ has a Fubini Property, if for every Borel set $B\se X\times Y$
	\[
		\{x\in X: B_x\notin\Jj\}\in \Ii \IMP \{y\in Y: B^y\notin X\}\in \Jj.
	\]
	If $(\Ii, \Ii)$ has a Fubini Property, then we will simply say that $\Ii$ has it. Let us notice that Kuratowski-Ulam Theorem and Fubini Theorem imply that $\Mm$ and $\Nn$ respectively possess the Fubini Property.
\begin{theorem}[\cite{Zeberski}, Thm 4 and Thm 5] Let $\RR\es\Ii\in\{\Mm, \Nn\}$ and $G\se \RR^2$ be a Borel set such that $G^c\in \Ii\otimes \Ii.$ Then there are two ets $B,P\se\RR$ such that $B\times P \subseteq G$, $B^c\in \Ii$ and $P\in\Perf(\RR)$.
\end{theorem}
\begin{proof} Let $V$ be a universe of ZFC such that $G\in V$ and let $V'$ be its extension such that $\w_2 < add(\Ii).$ Let $b\in\w^\w\cap V$ be a Borel code for $G.$ Let $G^\star$ be a Borel subset of $\bbr^2$ decoded by $b$ in $V'.$ By the absolutness of Borel codes of sets from $\Ii$ it is the case that ${B^\star}^c$ is in $\Ii$ in $V'.$
\\
We work in $V'$ universe. Let $Z=\{ x\in \bbr:\; G_x^{\star c} \in \Ii\}$. By the Fubini Property $Z^c\in \Ii$. Then $|Z| = \c\ge \w_3.$ Let us choose any set $T\se Z$ of cardinality $\omega_2$. Since $\w_2 < \add(\Ii)$, the complement of a set $\bigcap_{t\in T} G^\star_t$ is in $\Ii.$ Let $B\in \Bor(\bbr)$ such that $B^c\in \Ii$, $B\subseteq \bigcap_{t\in T} G^\star_t$ and consider a set
	$
	A = \{ x\in\bbr:\; B\subseteq G^\star_x \}.
	$
	Clearly, $A$ is coanalytic. Since $T$ has a size $\w_2$ and $T\subseteq A$,  $A$ contains a perfect subset $P.$ It implies that $V'$ is a model for the following formula
	\[
	(\exists B\in \Bor(\bbr))(\exists P\in \Perf(\RR))(\forall x,y\in\bbr)((x,y) \in B\times P \IMP (x,y)\in G^\star). 
	\]
	It is $\Sigma^1_2$, hence by Shoenfield Absoluteness Theorem it also holds in $V$.
\end{proof}

%

\end{document}